\documentclass[a4paper,12pt, reqno]{amsart}

\usepackage{amsmath, amsthm, amscd, amsfonts, amssymb, graphicx, color}
\usepackage[bookmarksnumbered, colorlinks, plainpages]{hyperref}

\usepackage{enumitem}
\usepackage{url}
\usepackage{multirow}
\usepackage{graphicx}
\usepackage{subfigure}
\usepackage[pass]{geometry}
\usepackage{cite}
\usepackage{cancel}
\usepackage{color}
\usepackage{pgf,tikz}
\usetikzlibrary{arrows}
\usepackage{lscape}
\tikzstyle{v} = [circle, draw, inner sep=2pt, minimum size=3pt, fill=black]
\usetikzlibrary{calc,shapes,decorations.pathreplacing,matrix}
\usetikzlibrary{patterns}

\tikzset{square matrix/.style={
    matrix of nodes,
    column sep=-\pgflinewidth, row sep=-\pgflinewidth,
    nodes={draw,
      minimum height=4.5pt,
      anchor=center,
      text width=4.5pt,
      align=center,
      inner sep=0pt
    },
  },
  square matrix/.default=1.2cm
}

\newtheorem{Theorem}{Theorem}[section]
\newtheorem{Definition}[Theorem]{Definition}

\newtheorem{Proposition}[Theorem]{Proposition}
\newtheorem{Corollary}[Theorem]{Corollary}
\newtheorem{Remark}[Theorem]{Remark}

\newtheorem{Problem}[Theorem]{Problem}

\DeclareMathOperator{\leaf}{leaf}
\DeclareMathOperator{\diam}{diam}

\begin{document}

\title{Total outer-connected domination number of middle graphs}

\author[F. Kazemnejad]{Farshad Kazemnejad}
\address{Farshad Kazemnejad, Department of Mathematics, Faculty of Basic Sciences, Ilam University, P.O. Box 69315-516, Ilam, Iran.}
\email{kazemnejad.farshad@gmail.com}
\author[B. Pahlavsay]{Behnaz Pahlavsay}
\address{Behnaz Pahlavsay, Department of Mathematics, Hokkaido University, Kita 10, Nishi 8, Kita-Ku, Sapporo 060-0810, Japan.}
\email{pahlavsayb@gmail.com}
\author[E. Palezzato]{Elisa Palezzato}
\address{Elisa Palezzato, Department of Mathematics, Hokkaido University, Kita 10, Nishi 8, Kita-Ku, Sapporo 060-0810, Japan.}
\email{palezzato@math.sci.hokudai.ac.jp}
\author[M. Torielli]{Michele Torielli}
\address{Michele Torielli, Department of Mathematics and Statistics, Northern Arizona University, 801 S Osborne Drive
Flagstaff, AZ 86011, USA.}
\email{michele.torielli@nau.com}

\date{\today}

\begin{abstract}
 In this paper, we study the total outer-connected domination number of the middle graph of a simple graph and we obtain tight bounds for this number in terms of the order of the middle graph. We also compute the total outer-connected domination number of some families of graphs, explicitly. Moreover, some Nordhaus-Gaddum-like relations are presented for the total outer-connected domination number of middle graphs.
\\[0.2em]

\noindent
Keywords: Total Outer-Connected Domination number, Total Domination number, Middle graph, Nordhaus-Gaddum-like relation.
\\[0.2em]

\noindent 
\end{abstract}
\maketitle
\section{\bf Introduction}

Consider $G=(V(G),E(G))$ a simple graph. A \emph{dominating set} of $G$ is a set $D\subseteq V(G)$
such that  $N_G[v]\cap D\neq \emptyset$, for all $v\in V(G)$. The \emph{domination number} of $G$ is the minimum cardinality of a dominating set of $G$ and it is denoted by $\gamma(G)$. 
 Similarly, a \emph{total dominating set} of $G$ is a set $D\subseteq V(G)$ such that for each $v\in V(G)$, $N_G(v)\cap
D\neq \emptyset$. The \emph{total domination number $\gamma_t(G)$} of $G$ is the minimum cardinality of a total dominating set of $G$.

The previous two notions play a central role in graph theory and the literature on the subject is vast, see for example \cite{HHS5,HHS6,HeYe13}, \cite{farsom}, \cite{3totdominrook} and \cite{dominLatin}. Since the introduction of the domination number of a graph, many variation have been introduced. The one pertinent to this paper is the following.

 \begin{Definition}[\cite{tocds1}] 
 A set $D \subseteq V(G)$ is a total outer-connected dominating set of $G$ if $D$ is total dominating, and the induced subgraph $G[V(G)-D]$ is a connected graph. The total outer-connected domination number of $G$, denoted by $\gamma _{tc}(G)$, is the minimum cardinality of a total outer-connected dominating set of $G$.
\end{Definition}

The concept of total outer-connected domination has also been studied in \cite{tocds2,tocds3}.

Following our previous work \cite{KPPT2022connect, DSFBME, DSFBME2, KPPT2021color}, in this paper we study the total outer-connected domination number of the middle graph of a simple graph.


\begin{Definition}[\cite{HamYos}]
 The middle graph $M(G)$ of a graph $G$ is the graph whose vertex set is $V(G)\cup E(G)$ and two vertices $x, y$ in the vertex set of $M(G)$ are adjacent in $M(G)$ in case one the following holds
 \begin{enumerate}
 \item $x, y\in E(G)$ and $x, y$ are adjacent in $G$;
 \item $x\in V (G)$, $y\in E(G)$, and $x, y$ are incident in $G$. 
 \end{enumerate}
 \end{Definition}
 

In order to avoid confusion throughout the paper, we fix a ``standard'' notation for the vertex set and the edge set of the middle graph $M(G)$. Assume $V(G)=\{v_1,\dots, v_n\}$, then we set $V(M(G))=V(G)\cup \mathcal{M}$, where $\mathcal{M}=\{m_{ij}~|~ v_iv_j\in E(G)\}$ and $E(M(G))=\{v_im_{ij},v_jm_{ij}~|~ v_iv_j\in E(G)\}\cup E(L(G)) $, where $L(G)$ is the line graph of $G$.


The paper proceeds as follows. In Section 2, first we present some upper and lower bounds for $\gamma _{tc}(M(G))$ in terms of the order of the graph $G$, and we compute explicitly $\gamma _{tc}(M(G))$ for several known families of graphs.  In Section 3, we study the total outer-connected domination number of the middle graph of trees and we compute explicitly $\gamma _{tc}(M(G))$ for several known families of trees.
 In Section 4, we compute the total outer-connected domination number of the middle graphs of some graphs obtained using operation on graphs. Finally, in the last Section we present some Nordhaus-Gaddum like relations for the total outer-connected domination number of middle graphs.
 
 Throughout this paper, we use standard notation for graphs and we assume that each graph is non-empty, finite, undirected and simple. We refer to \cite{bondy2008graph} as a general reference on graph theory. 

\section{\bf Total outer-connected domination number of middle graphs}

We first recall the following result from \cite{DSFBME2} which is useful for our investigation.
\begin{Theorem}\label{corol:mintotdominfamily} 
	Let $G$ be a graph with $n\ge3$ vertices. If $G$ has a subgraph isomorphic to $P_n$, then $$\gamma_t(M(G))=\lceil\frac{2n}{3}\rceil.$$
\end{Theorem}

We start our study of the total outer-connected domination number by describing a lower and an upper bound.

\begin{Theorem}\label{theo:TOCD<n+m-1}
Let $G$ be a connected graph with $n\ge3$ vertices. Then $$2|\leaf(G)| \le\gamma _{tc}(M(G))\le n+m-1$$
where $m$ is the size of $G$ and $\leaf(G)=\{v\in V(G)~|~\deg_G(v)=1\}$.
\end{Theorem}
\begin{proof} 
	Assume $V(G)=\{v_1, \dots, v_n\}$. Then $V(M(G))=V(G) \cup \mathcal{M}$ where $\mathcal{M}=\{m_{ij}~|~ v_iv_j\in E(G)\}$. Since if $v \in V(G)$, 
	then $D=V(M(G))\setminus \{v\}$ is a total outer-connected dominating set of $M(G)$, we have $$\gamma _{tc}(M(G))\le |\mathcal{M}|+|V(G)|-1= n+m-1.$$
	\vspace{0.02cm}
	On the other hand, let $D$ be  a total outer-connected dominating set of $M(G)$ and $\mathcal{D}_{1}=\{m_{ij}~|~ v_i \in \leaf(G)~ \mbox{ or }~ v_j \in \leaf(G)\}$. Since $D$ is a total dominating set of $M(G)$, then $\mathcal{D}_{1} \subseteq D$. Obviously, $\leaf(G) \subseteq D$ because otherwise $M(G)-D$ is disconnected, which is a contradiction. Hence
	$$|D| \ge |\mathcal{D}_{1}|+|\leaf(G)|=|\leaf(G)|+|\leaf(G)|=2|\leaf(G)|.$$
	This implies that 
	$\gamma _{tc}(M(G))\ge 2|\leaf(G)|.$
\end{proof}
In the next result we calculate the total outer-connected dominating set of the middle graph of a cycle graph.
\begin{Theorem}\label{tOCDcycle}
	For any cycle $C_n$ of order $n \geq 3$,  $$\gamma _{tc}(M(C_n))=2n-3.$$
\end{Theorem}
\begin{proof}
	To fix the notation, assume that $V(C_n)=V=\{v_1, \dots, v_n\}$ and $E(C_n)=\{v_1v_2, v_2v_3,\dots,v_{n-1}v_n,v_nv_1\}$. Then $V(M(C_n))=V(C_n)\cup \mathcal{M}$, where $\mathcal{M}=\{ m_{i(i+1)}~|~1\leq i \leq n-1 \}\cup\{m_{1n}\}$. Let $D$ be a total outer-connected dominating set of $M(C_{n})$.
	
	
	Suppose that $|\mathcal{M} \cap D| = n-2$, and let  $m_{ij}, m_{pq} \in \mathcal{M} $ such that $m_{ij}, m_{pq} \notin D $ for some $i,j,p,q$.
	If $m_{ij}$ and $m_{pq}$ are adjacent in $M(G)$, without loss of generality, we can assume that $m_{ij}=m_{i(i+1)}$ and $m_{pq}=m_{(i+1)(i+2)}$ for some $i$. This implies that $N_{M(G)}(v_{i+1})\cap D=\emptyset$, which is a contradiction. If  $m_{ij}$ and $m_{pq}$ are non-adjacent in $M(G)$, then $m_{ij}$ and $m_{pq}$ are non-adjacent in $M(G)-D$, and there is not any path between $m_{ij}$ and $m_{pq}$  in $M(G)-D$, and so $M(G)-D$ is disconnected, which is a contradiction. This implies that $n\ge |\mathcal{M} \cap D| \ge n-1.$
	
	If $| \mathcal{M} \cap D| = n$, since $V$ is an independent set in $M(G)$ and $M(G)-D$ is connected, then we have that $|V \cap D| \ge n-1 > n-2.$ This implies that
$$|D| \ge |V \cap D|+|\mathcal{M} \cap D| \ge 2n-1>2n-3.$$

Similarly, if $| \mathcal{M} \cap D| = n-1$, then there exists a vertex $m_{i(i+1)} \in \mathcal{M}$ such that  $m_{i(i+1)} \notin D$. Assume that $|V \cap D| < n-2$. Then there exist at least three vertices $v_p, v_q, v_r \in V$ such that $v_p, v_q, v_r \notin D$. Since $N_{M(G)}(m_{i(i+1)})\cap V=\{v_i , v_{i+1}\}$, then there exists at least a vertex $v \in \{v_p, v_q, v_r\}$ such that $v \notin D$ and there is no path between $v$ and $m_{i(i+1)}$ in $M(G)-D$, which is a contradiction. This shows that $|V \cap D| \ge n-2.$ 

Putting the previous two inequalities together we obtain that
$$|D| \ge |V \cap D|+|\mathcal{M} \cap D| \ge 2n-3.$$
Now since $D=(V \cup \mathcal{M}) \setminus \{v_1, v_2, m_{12}\}$ is  a total-outer connected dominating set of $M(G)$ with $|D|=2n-3$, we have $\gamma _{tc}(M(G))=2n-3.$
	
\end{proof}

In the next result, we calculate the total outer-connected domination number of the middle graph of a complete graph. Since $K_3$ is isomorphic to $C_{3}$, then $\gamma _{tc}(M(K_{3}))=\gamma _{tc}(M(C_{3}))=3$ by Theorem \ref{tOCDcycle}.
\begin{Theorem}\label{tocdcompletegraph}
	For any complete graph $ K_{n}$ of order $n \geq 4$, we have
	$$\gamma _{tc}(M(K_{n}))=\lceil 2n/3\rceil$$
\end{Theorem}
\begin{proof}
	To fix the notation, assume that $V(M(K_n))=V(K_n) \cup \mathcal{M}$ where $V(K_n)=\{v_1, \dots, v_n\}$ and $\mathcal{M}=\{m_{ij} ~|~v_i v_j \in E(K_n)\}$. Let $D$ be a total outer-connected dominating set of $M(K_n)$. By Theorem \ref{corol:mintotdominfamily} $\gamma _{tc}(M(K_{n}))\ge \lceil 2n/3\rceil$.
Set $$S=\{m_{(3i+1)(3i+2)}, m_{(3i+2)(3i+3)}~|~ 0 \le i \le \lfloor n/3 \rfloor-1\}.$$
Now since
$$D= 
\begin{cases}
S & \mbox{if }n\equiv 0 \pmod{3},\\ 
S \cup \{m_{(n-1)n}\} & \mbox{if }n\equiv 1 \pmod{3},\\ 
S \cup \{m_{(n-2)(n-1)},m_{(n-1)n}\}  & \mbox{if }n\equiv 2 \pmod{3}.
\end{cases}$$
 is a total-outer connected dominating set of $M(K_n)$ with $|D|=\lceil 2n/3\rceil$ we have $ \gamma_{tc}(M(K_n)) = \lceil 2n/3\rceil.$
\end{proof}
In the next result, we calculate the total outer-connected domination number of the middle graph of a wheel graph. Since $W_4$ is isomorphic to $K_{4}$, then $\gamma _{tc}(M(W_{4}))=\gamma _{tc}(M(K_{4}))=4$ by Theorem \ref{tocdcompletegraph}.
\begin{Theorem} \label{tocd(M(W))} For any wheel $W_n$ of order $n\geq 5$, 
	$$ \gamma_{tc}(M(W_n)) = \lceil 2n/3\rceil.$$
\end{Theorem}
\begin{proof}
	To fix the notation, assume $V(W_n)=V=\{v_0,v_1,\dots, v_{n-1}\}$ and $E(W_n)=\{v_0v_1,v_0v_2,\dots, v_0v_{n-1}\}\cup\{v_1v_2, v_2v_3,\dots,v_{n-1}v_1\}$. Then we have $V(M(W_n))=V(W_n)\cup \mathcal{M}$, where $\mathcal{M}=\{ m_{0i}~|~1\leq i \leq n-1 \}\cup\{ m_{i(i+1)}~|~1\leq i \leq n-1 \}\cup\{m_{1(n-1)}\}$.
By Theorem \ref{corol:mintotdominfamily}, $\gamma _{tc}(M(W_{n}))\ge \lceil 2n/3\rceil$.
Set $$S=\{m_{(3i+1)(3i+2)}, m_{(3i+2)(3i+3)}~|~ 0 \le i \le \lfloor n/3 \rfloor-1\}.$$
Now since
$$D= 
\begin{cases}
S & \mbox{if }n\equiv 0 \pmod{3},\\ 
S \cup \{m_{(n-1)n}\} & \mbox{if }n\equiv 1 \pmod{3},\\ 
S \cup \{m_{(n-2)(n-1)},m_{(n-1)n}\}  & \mbox{if }n\equiv 2 \pmod{3}.
\end{cases}$$
is a total-outer connected dominating set of $M(W_n)$ with $|D|=\lceil 2n/3\rceil$, we have $ \gamma_{tc}(M(W_n)) = \lceil 2n/3\rceil.$
\end{proof}
As an immediate consequence of Theorems \ref{corol:mintotdominfamily}, \ref{tocdcompletegraph} and \ref{tocd(M(W))}, we have the following result.
\begin{Corollary}
 For any $n \ge 4$, there exists a connected graph $G$ of order $n$ with $$ \gamma_{tc}(M(G)) = \gamma_{t}(M(G)).$$
\end{Corollary}	
In the next theorem, we calculate the total outer-connected domination number of the middle graph of a complete bipartite graph $K_{n_1,n_2}$. If $n_1=1$, we will compute $\gamma_{tc}(M(K_{1,n_2}))$ in Corollary \ref{tocdK1n=2n-2}. Moreover, since $K_{2,2}$ is isomorphic to $C_{4}$, we have $\gamma _{tc}(M(K_{2,2}))=\gamma _{tc}(M(C_{4}))=5$ by Theorem \ref{tOCDcycle}.
\begin{Theorem}\label{toutcompletebipartitegr}
	Let $K_{n_1,n_2}$ be the complete bipartite graph with $n_2\geq n_1 \geq 2$ and $(n_1,n_2)\neq (2,2)$. Then
	$$\gamma_{tc}(M(K_{n_1,n_2}))=
	\begin{cases}
	n_{2}+3  & \mbox{if } ~n_1= 2,\\ 
	5  & \mbox{if }  ~~(n_1,n_2)= (3,3),\\ 
	n_2  & \mbox{otherwise}.
	\end{cases}$$
\end{Theorem}
\begin{proof}  To fix the notation, assume $V(K_{n_1,n_2})=\{v_1,\dots, v_{n_1},u_1,\dots, u_{n_2}\}$ and $E(K_{n_1,n_2})=\{v_iu_j ~|~1\leq i \leq n_1, 1\leq j \leq n_2\}$. Then $V(M(K_{n_1,n_2}))=V(K_{n_1,n_2})\cup \mathcal{M}$, where
	$\mathcal{M}=\{ m_{ij}~|~1\leq i \leq n_1, 1\leq j \leq n_2 \}. $
	Let $D$ be a total outer-connected dominating set of $M(K_{n_1,n_2})$. 	Since $D$ is a total dominating set for $M(K_{n_1,n_2})$ and $N_{M(K_{n_1,n_2})}[u_i] \cap N_{M(K_{n_1,n_2})}[u_j]=\emptyset$ for every $i,j \in \{1,\dots,n\}$, this implies that $\gamma_{tc}(M(K_{n_1,n_2}))\ge n_2$.
	
	 First, assume that $n_1=2$. Since $D=\{m_{11},u_{1}, v_{2}\}\cup\{m_{2i}~|~ 1\leq i \leq n_2\}$ is an total outer-connected dominating set of $M(K_{2,n_2})$ with $|D|=n_2+3$, This implies that $\gamma_{tc}(M(K_{2,n_2}))\le n_2+3$. We show that  $\gamma_{tc}(M(K_{2,n_2}))\ge n_2+3$.
	
	Set $\mathcal{M}_{i}=\{m_{ij}~|~1\leq j \leq n_2\}$ for $i=1,2$. Then $N_{M(K_{2,n_2})}[v_i] \cap D \neq \emptyset$ implies that $|\mathcal{M}_{i} \cap D| \ge 1$ for $i=1,2$. Also $N_{M(K_{2,n_2})}[u_j] \cap D \neq \emptyset$ implies that $|\mathcal{M} \cap D| \ge n_2$ for $1\leq j \leq n_2$. We claim that $|D| >n_2+2 $. Assume that $|D| \le n_2+2 $. Then $(|\mathcal{M}_{1} \cap D|, |\mathcal{M}_{2} \cap D|)=(x,y)$ where $x,y \ge 1$ and $n_2 \le x+y \le n_2+2$. 
	We consider three cases.
	
	Assume $x+y=n_{2}$. Then $|\mathcal{M} \cap D| =n_2$. Set $\mathcal{M}^{'}_{1}=D \setminus \mathcal{M}_{1}$, $\mathcal{M}^{'}_{2}=D \setminus \mathcal{M}_{2}$, $U_1=\{u_{t}~|~m_{1t} \in \mathcal{M}_{1} \}$ and $U_2=\{u_{k}~|~m_{2k} \in \mathcal{M}_{2}\}$. Now set $H_1=G[\mathcal{M}^{'}_{2} \cup U_1 \cup \{v_2\} ]$ and $H_2=G[\mathcal{M}^{'}_{1} \cup U_2 \cup \{v_1\} ]$. Then $H_1 , H_2 \subseteq M(K_{2,n_2})-D$ and $V(H_1) \cap V(H_2)=\emptyset $. This implies that $M(K_{2,n_2})-D$ is disconnected, which is a contradiction.
	
	Assume $x+y=n_{2}+1$. Then $|\mathcal{M} \cap D| =n_2+1$. This implies that $|D \cap \{v_1 , v_2\}| \le 1$. We consider three cases. First, let $v_1 , v_2 \notin D$. Then $N_{M(K_{2,n_2})}[u_j] \cap D \neq \emptyset$ for $1\leq j \leq n_2$ implies that there are not any path between $v_1$ and $v_2$ in $M(K_{2,n_2})-D$ which is a contradiction. Assume that  $v_1 \in D$ and  $v_2  \notin D$. Then $|U \cap D|=0$ and $y \ge 1$ implies that there exists at least a vertex $m_{2t} \in D$ for some $1\leq t \leq n_2$ such that there are not any path between $u_1$ and $u_t$ in $M(K_{2,n_2})-D$ which is a contradiction. Now let $v_1 \notin D$ and  $v_2  \in D$. Then $|U \cap D|=0$ and obviously there are not any path between  $v_1$ and $u_1$ in $M(K_{2,n_2})-D$ which is a contradiction.
	
	Assume $x+y=n_{2}+2$. Then $|\mathcal{M} \cap D| =n_2+2$. $\{m_{1t}, m_{2t}\} \subseteq N_{M(K_{2,n_2})}[u_t] \cap D \neq \emptyset$ implies that there are not any path between  $v_1$ and $v_2$ in $M(K_{2,n_2})-D$ which is a contradiction. Therefore $\gamma_{tc}(M(K_{2,n_2}))\ge n_2+3$.
	
	Now let $(n_1,n_2)= (3,3)$. By Theorem \ref{corol:mintotdominfamily} $\gamma_{t}(M(K_{3,3}))=4$. This implies that $\gamma_{tc}(M(K_{3,3})) \ge 4$. Assume that $|D|= 4$. Set $\mathcal{M}_{i}=\{m_{ij}~|~1 \le j \le 3\}$ for  $1 \le i \le 3$. We claim that $1 \le | \mathcal{M}_{i} \cap D| \le 2$.  $N_{M(K_{3,3})}(v_{i}) \cap D \neq \emptyset$ for $1 \le i \le 3$ implies that $| \mathcal{M}_{i} \cap D| \ge 1$. Now suppose that $| \mathcal{M}_{i} \cap D|=3$ for some $i$.  Since for every $j\neq i$, $| \mathcal{M}_{j} \cap D| \ge 1$, this implies that $|D| \ge 5$,  which is a contradiction. Now by assumption $|D| = 4$ and $1 \le | \mathcal{M}_{i} \cap D| \le 2$ for $1 \le i \le 3$, we conclude that there exsists a $\mathcal{M}_{i}$ we say $\mathcal{M}_{1}$ such that $| \mathcal{M}_{1} \cap D|=2$. This implies that $| \mathcal{M}_{2} \cap D|=|\mathcal{M}_{3} \cap D|=1$ and so $M(K_{3,3})-D$ is disconnected, which is a contradiction. Therefore $|D| \ge 5$.
	
	Since $D=\{ v_{1},m_{11},m_{22},m_{23},m_{33},\}$ is an total outer-connected dominating set of $M(K_{3,3})$ with $|D|=5=n_2+2$, This implies that $\gamma_{tc}(M(K_{3,3}))=5$. 
	\vskip 0.02cm
	 Finally, let $n_1\neq 2$ and $(n_1,n_2)\neq (3,3)$. Set $D_{1}=$ $$\{m_{(3i+1)(3i+1)}, m_{(3i+1)(3i+2)},m_{(3i+2)(3i+3)},m_{(3i+3)(3i+3)},~|~ 0 \le i \le \lfloor n/3 \rfloor-1\}$$
	 and $D_{2}=\{m_{(n_1)(n_1+i)}~|~ 1\leq i \leq n_2-n_{1}\}$.
	 \vskip 0.02cm
	 Assume that $n_1= n_2$. Then 
	 $$D= 
	 \begin{cases}
	 D_{1} & \mbox{if }n_1\equiv 0 \pmod{3},\\ 
	 D_{1} \cup \{m_{(n_1-1)n_1},m_{n_1 n_1} \} & \mbox{if }n_1\equiv 1 \pmod{3},\\ 
	 D_{1} \cup \{m_{(n_1-1)(n_1-1)},m_{(n_1-1)n_1},m_{n_1 n_1} \}  & \mbox{if }n_1\equiv 2 \pmod{3}.
	 \end{cases}$$
	 is a total outer-connected dominating set of $M(K_{n_1,n_2})$ with $|D|=n_2$. This implies that $\gamma_{tc}(M(K_{n_1,n_2}))= n_2$.

Assume that $n_1+1 \le  n_2 \le 2n_{1}-1$. Then 
$$D= 
\begin{cases}
D_{1} \cup D_{2} & \mbox{if }n_1\equiv 0 \pmod{3},\\ 
D_{1}  \cup D_{2} \cup \{m_{(n_1-1)n_1},m_{n_1 n_1} \} & \mbox{if }n_1\equiv 1 \pmod{3},\\ 
D_{1} \cup D_{2} \cup \{m_{(n_1-1)(n_1-1)},m_{(n_1-1)n_1},m_{n_1 n_1} \}  & \mbox{if }n_1\equiv 2 \pmod{3}.
\end{cases}$$
is a total outer-connected dominating set of $M(K_{n_1,n_2})$ with $|D|=n_2$. This implies that $\gamma_{tc}(M(K_{n_1,n_2}))= n_2$.

Assume now that $  n_2 \ge 2n_{1}$. Then 
$$D= \{m_{ii}, m_{i(n_{1}+i)} ~|~ 1\leq i \leq n_{1}\} \cup \{ m_{n_{1}(2n_{1}+i)} ~|~ 1\leq i \leq n_{2}-2n_{1}\}
$$
is a total outer-connected dominating set of $M(K_{n_1,n_2})$ with $|D|=n_2$. This implies that $\gamma_{tc}(M(K_{n_1,n_2}))= n_2$.
\end{proof}
\section{The middle graph of a tree}
If we assume that the graph is a tree, we can easily rewrite Theorem \ref{theo:TOCD<n+m-1} and obtain the following result.
\begin{Corollary}\label{leaf(T) le gamma _{tc}(M(T)) le 2n-2}
Let $T$ be a tree with $n \ge 2$ vertices. Then
	 $$2|\leaf(T)| \le\gamma _{tc}(M(T))\le 2n-2.$$
\end{Corollary}
As a consequence of Corollary \ref{leaf(T) le gamma _{tc}(M(T)) le 2n-2}, we have the following result.
\begin{Corollary}\label{tocdK1n=2n-2}
	For any star graph $ K_{1,n-1}$ with $n \geq 3$, we have
	$$\gamma _{tc}(M(K_{1,n-1}))=2n-2.$$
\end{Corollary}	

\begin{Remark}
 By Corollary \ref{tocdK1n=2n-2}, the inequalities described in Corollary \ref{leaf(T) le gamma _{tc}(M(T)) le 2n-2} are sharp.	
\end{Remark}

In the case of a tree $T$ that is not isomorphic to $ K_{1,n-1}$, we can describe an upper bound for $\gamma _{tc}(M(T))$.
\begin{Theorem}\label{T neq K_{1,n-1}}
Let $T \neq K_{1,n-1} $ be a tree of order $n \ge 4$. Then 
$$\gamma _{tc}(M(T)) \le 2n-4.$$
\end{Theorem}
\begin{proof} 
	Let $T \neq K_{1,n-1} $ be a tree of order $n \ge 4$ with $V(T)=\{v_1, \dots, v_n\}$. Then $V(M(T))=V(T) \cup \mathcal{M}$ where $\mathcal{M}=\{m_{ij}~|~ v_iv_j\in E(T)\}$. 
	$T \neq K_{1,n-1} $ implies that $P_4 \subseteq T$ as an induced subgraph. Without loss of generality, let $P_4 : v_1v_2v_3v_4$ be an induced subgraph in $T$. Then $P_7 : v_1 m_{12}v_2 m_{23} v_3 m_{34} v_4$  is an induced subgraph in $M(T)$. Now since $D=V(M(T)) \setminus \{v_2, m_{23}, v_3\}$ is a total outer-connected dominating set of $M(T)$, we have $\gamma _{tc}(M(T)) \le 2n-4.$
\end{proof}
By Corollary \ref{tocdK1n=2n-2} and Theorem \ref{T neq K_{1,n-1}} we have the following result.
\begin{Corollary}
Let $T $ be a tree of order $n \ge 4$. Then 
$\gamma _{tc}(M(T)) = 2n-2$ if and only if $T = K_{1,n-1}$.
\end{Corollary}
\begin{Remark}
For any tree $T$ of order $n \ge 4$, $\gamma _{tc}(M(T)) \neq 2n-3$.
\end{Remark}
The next theorem computes the total outer-connected domination number of the middle graph of a path of order. Clearly, $\gamma _{tc}(M(P_{2}))=2$. Moreover, $P_3$ is isomorphic to $K_{1,2}$ and so $\gamma _{tc}(M(P_{3}))=\gamma _{tc}(M(K_{1,2}))=4$ by Corollary \ref{tocdK1n=2n-2}.
\begin{Theorem}\label{tocdpn=2n-4}
	For any path $ P_{n}$ of order $n \geq 4$, we have
	$$\gamma _{tc}(M(P_{n}))=2n-4.$$
\end{Theorem}	
\begin{proof}
	Let   $G= P_{n}$ be a path of order $n \geq 4$ with vertex set  $V(G)=\{v_1, \dots, v_n\}$. Then $V(M(G))=V(G) \cup \mathcal{M}$ where $\mathcal{M}=\{m_{i(i+1)}~|~ 1 \leq i \leq n-1\}$. Let $D$ be a total-outer connected dominating set of $M(G)$. 
	
	%
	Since $D$ is a total dominating set of $M(G)$ then $ N_{M(G)}(v_1)\cap D=\{m_{12}\}$ and $ N_{M(G)}(v_n)\cap D=\{m_{(n-1)n}\}$, and hence $m_{12}, m_{(n-1)n} \in D$. 
	
	Assume now there exist $m_{ij}, m_{pq} \in \mathcal{M}_{1}=\mathcal{M} \setminus \{m_{12}, m_{(n-1)n}\}$ such that $m_{ij}, m_{pq} \notin D $ for some $i,j,p,q$.
	If $m_{ij}$ and $m_{pq}$ are adjacent in $M(G)$, then, without loss of generality, we can assume that $m_{ij}=m_{i(i+1)}$ and $m_{pq}=m_{(i+1)(i+2)}$ for some $2\le i\le n-3$. This implies that $N_{M(G)}(v_{i+1})\cap D=\emptyset$, which is a contradiction. If $m_{ij}$ and $m_{pq}$ are non-adjacent in $M(G)$, then $m_{ij}$ and $m_{pq}$ are non-adjacent in $M(G)-D$, and so $M(G)-D$ is disconnected, which is a contradiction. This shows that 
	$n-2 \le| \mathcal{M} \cap D| \le n-1.$
	
	%
	
	Assume $| \mathcal{M} \cap D| = n-1$. Since $V$ is an independent set in $M(G)$ and $M(G)-D$ is connected, this implies that $|V \cap D| \ge n-1.$ this implies that
	$$|D| \ge |V \cap D|+|\mathcal{M} \cap D| \ge 2n-2\ge 2n-4.$$
	
	Now suppose that $| \mathcal{M} \cap D| = n-2$. Then there exists a vertex $m_{i(i+1)} \in \mathcal{M}$ for some $i \neq 1,n-1$ such that  $m_{i(i+1)} \notin D$. Assume that $|V \cap D| < n-2$. Then there exist at least three vertices $v_p, v_q, v_r \in V(G)$ such that $v_p, v_q, v_r \notin D$. Since $N_{M(G)}(m_{i(i+1)})\cap V(G)=\{v_i , v_{i+1}\}$, then there exists at least a vertex $v \in \{v_p, v_q, v_r\}$ such that $v \notin D$ and there is no path between $v$ and $m_{i(i+1)}$ in $M(G)-D$, which is a contradiction. This implies that $|V \cap D| \ge n-2.$ 
	
	Putting the previous two inequalities together we obtain that
	$$|D| \ge |V \cap D|+|\mathcal{M} \cap D| \ge 2n-4.$$
	On the other hand,$\gamma _{tc}(M(G)) \le 2n-4$ by Theorem \ref{T neq K_{1,n-1}}. We have $\gamma _{tc}(M(G)) = 2n-4$.
\end{proof}
\begin{Remark}
Upper bound in Theorem \ref{T neq K_{1,n-1}} is tight by Theorem \ref{tocdpn=2n-4}.
\end{Remark}

It is sufficient to add some assumptions on the diameter of a tree $T$, to compute $\gamma _{tc}(M(T))$ explicitly.
\begin{Proposition}\label{gamma _{tc}(M(T))=doubstar=2n-4}
 Let $T$ be a tree of order $n \ge 4$ with $\diam(T) = 3$. Then	
	$\gamma _{tc}(M(T))=2n-4.$
\end{Proposition}
\begin{proof}
The assumption that $\diam(T) = 3$ implies that $T$ is a tree which is obtained by joining central vertex $v$ of $K_{1,p}$ and the central vertex $w$ of $K_{1,q}$ where $p + q = n - 2$. Let $\leaf(T)=\{v_{i}~|~ 1 \le i \le n-2\}$  be the set of leaves of $T$. Obviously $V (T) = \leaf(T) \cup \{v, w\} $ and 
$| \leaf(T)| = n - 2$. Define $v_{n-1} = v$ and $v_{n} = w$. Then $V(M(T))=\{v_{i}~|~1 \le i \le n\} \cup \{m_{i(n-1)}~|~1 \le i \le p\} \cup \{m_{i(n)}~|~p+1 \le i \le n-2\} \cup \{m_{(n-1)n}\}$. By Corollary \ref{leaf(T) le gamma _{tc}(M(T)) le 2n-2}, $\gamma _{tc}(M(T)) \ge 2|\leaf(T)|  =2n-4$.
\vskip 0.02cm
Since $V(M(T)) \setminus \{v_{n-1},v_{n},m_{(n-1)n}\}$ is a total outer-connected dominating set of $M(T)$ with $2n-4$ vertices, then $\gamma_{tc}(M(T))\le 2n-4$. 
\end{proof}
By Corollary \ref{tocdK1n=2n-2} and Proposition \ref{gamma _{tc}(M(T))=doubstar=2n-4}, we have the following result.
\begin{Corollary}\label{ gamma _{tc}(M(T))=2leaf(T)}
	Let $T$ be a tree of order $n \ge 4$ with $2 \le \diam(T) \le 3$. Then
	$\gamma _{tc}(M(T))=2|\leaf(T)|.$
\end{Corollary}
\begin{Proposition}\label{gamma _{tc}(M(T))=2n-6}
Let $T$ be a tree of order $n \ge 6$ with $ \diam(T)=4$ and $P_{5}:v_{1}v_{2}v_{3}v_{4}v_{5}$ be the longest path in $T$ such that $\deg_{T}(v_{3})=n-3$. Then $\gamma _{tc}(M(T))=2n-6$.
\end{Proposition}
\begin{proof}
Let $T$ be a tree of order $n \ge 6$ with $ \diam(T)=4$ and vertex set $V(T)=\{v_{1}, \dots,v_{n}\}$. Assuming that $P_{5}:v_{1}v_{2}v_{3}v_{4}v_{5}$ is the longest path in $T$ such that $\deg_{T}(v_{3})=n-3$, implies that $V(M(T))=V(T) \cup \{m_{i(i+1)}~|~1 \le i \le 4\} \cup \{m_{3(5+i)}~|~1 \le i \le n-5\} $. Obviously, $\leaf(T)=\{v_{1}, v_{5}\} \cup \{v_{5+i}~|~1 \le i \le n-5\}$ and $|\leaf(T)|=n-3$.  By Corollary \ref{leaf(T) le gamma _{tc}(M(T)) le 2n-2}, $\gamma _{tc}(M(T)) \ge 2|\leaf(T)|  =2n-6$.
\vskip 0.02cm
Now since $\{v_{1}, v_{5}, m_{12}, m_{45}\} \cup \{v_{5+i},m_{3(5+i)}~|~1 \le i \le n-5\}$ is a total outer-connected dominating set of $M(T)$ with $2n-6$ vertices, this implies that $\gamma_{tc}(M(T))\le 2n-6$.
\end{proof}
\begin{Remark}\label{123456}
By Proposition \ref{gamma _{tc}(M(T))=2n-6}, there exists a tree $T$ of order $n=6$ with $\gamma_{tc}(M(T))=n$.
\end{Remark}
Remark \ref{123456} suggests the following natural problem.
\begin{Problem}
Characterize the trees $T$ of order $n$ with $\gamma_{tc}(M(T))=n$.
\end{Problem}
\section{Operation on graphs}
	Consider a simple graph $G=(V(G), E(G))$, the \emph{corona} $G\circ K_1$ of $G$ is the graph of order $2|V(G)|$ obtained from $G$ by adding a pendant edge to each vertex of $G$. The \emph{$2$-corona} $G\circ P_2$ of $G$ is the graph of order $3|V(G)|$ obtained from $G$ by attaching a path of length $2$ to each vertex of $G$ so that the resulting paths are vertex-disjoint.
\begin{Theorem}\label{tocdcorona1}
	For any connected graph $G$ of order $n\geq 2$, 
	$$\gamma _{tc}(M(G\circ K_1))=2n.$$
\end{Theorem}
\begin{proof} 
	To fix the notation, assume $V(G)=\{v_1,\dots, v_n\}$. Then $V(G\circ K_1)= \{v_{1},\dots, v_{2n}\}$ and $E(G\circ K_1)=\{v_1v_{n+1},\dots, v_nv_{2n} \}\cup E(G) $. Then $V(M(G\circ K_1))=V(G\circ K_1)\cup \mathcal{M}$, where $\mathcal{M}=\{ m_{i(n+i)}~|~1\leq i \leq n \}\cup \{ m_{ij}~|~v_iv_j\in  E(G)\}$.
	
	By Theorem \ref{theo:TOCD<n+m-1}, $\gamma _{tc}(M(G\circ K_1)) \ge 2|\leaf(G\circ K_1)|=2n$.
	Now since $ D=\{ m_{i(n+i)}, v_{n+i}~|~1 \leq i\leq n\} $ is a total-outer connected dominating set of $M(G\circ K_1)$ with $|D| = 2n$, we have 
	$$\gamma _{tc}(M(G\circ K_1))=2n.$$
\end{proof}
As an immediate consequence of Theorem \ref{tocdcorona1}, we have the following result.
\begin{Corollary}\label{tocdGofordern=n}
	For any $n \ge 4$, there exists a connected graph $G$ of order $n$ such that
	$$\gamma _{tc}(M(G))=n.$$
\end{Corollary}
\begin{Theorem}\label{tocd2corona}
	For any connected graph $G$ of order $n\geq 2$, $$2n+\gamma_{t}(M(G)) \le \gamma_{tc}(M(G\circ P_2)) \le 4n.$$
\end{Theorem}
\begin{proof} 
	To fix the notation, assume $V(G)=\{v_1,\dots, v_n\}$. Then $V(G\circ P_2)= \{v_{1},\dots, v_{3n}\}$ and $E(G\circ P_2)=\{v_iv_{n+i}, v_{n+i}v_{2n+i}~|~1\leq i \leq n \}\cup E(G) $. Then $V(M(G\circ P_2))=V(G\circ P_2)\cup \mathcal{M}$, where $\mathcal{M}=\{ m_{i(n+i)},m_{(n+i)(2n+i)}~|~1\leq i \leq n \}\cup \{ m_{ij}~|~v_iv_j\in  E(G)\}$. 
	
	Let $D$ be a total-outer connected dominating set of $M(G\circ P_2)$. Since $N_{M(G\circ P_2)}(v_{2n+i})\cap D = \{m_{(n+i)(2n+i)}\}$ for every $1\leq i \leq n$, we have that $\{m_{(n+i)(2n+i)}~|~1\leq i \leq n\} \subseteq D$. This implies that $v_{2n+i} \in D$ for every $1\leq i \leq n$ because $M(G\circ P_2)-D$ is connected. Hence $\{m_{(n+i)(2n+i)}, v_{2n+i}~|~1\leq i \leq n\}=D_{1} \subseteq D$.   $N_{M(G\circ P_2)}(v)\cap D_{1} = \emptyset$ for every $v \in V(M(G))=\{v_i~|~1\leq i \leq n\} \cup \{m_{ij}~|~v_iv_j\in  E(G)\}$ implies that $|D| \ge 2n+\gamma_{t}(M(G))$, and hence we obtain the first inequality. 
	\vskip 0.04cm
	Since $D=\{v_{n+i}, v_{2n+i}, m_{i(n+i)}, m_{(n+i)(2n+i)} ~|~1\leq i \leq n\}$ is a total-outer connected dominating set of $M(G\circ P_2)$ with $|D|=4n$, we have $\gamma_{tc}(M(G\circ P_2)) \le 4n$, and we obtain the second inequality.
\end{proof}
\begin{Remark}
	The upper bound in Theorem \ref{tocd2corona} is tight when $G=P_2$. This is because $G\circ P_2$ is isomorphic to $P_6$, and hence $\gamma_{tc}(M(G\circ P_2))=\gamma_{tc}(M( P_6))=4n=2n-4=8$ by Theorem \ref{tocdpn=2n-4}.
\end{Remark}
\begin{Proposition}\label{tocdKnoP2corona}
	For any complete graph $K_n$ of order $n\geq 3$, $$ \gamma_{tc}(M(K_n\circ P_2)) = 2n+\lceil 2n/3\rceil.$$
\end{Proposition}
\begin{proof} 
	To fix the notation, assume $V(K_n)=\{v_1,\dots, v_n\}$. Then $V(K_n\circ P_2)= \{v_{1},\dots, v_{3n}\}$ and $E(K_n\circ P_2)=\{v_iv_{n+i}, v_{n+i}v_{2n+i}~|~1\leq i \leq n \}\cup E(K_n) $. Then $V(M(K_n\circ P_2))=V(K_n\circ P_2)\cup \mathcal{M}$, where $\mathcal{M}=\{ m_{i(n+i)},m_{(n+i)(2n+i)}~|~1\leq i \leq n \}\cup \{ m_{ij}~|~v_iv_j\in  E(K_n)\}$. 
	
	By Theorem \ref{corol:mintotdominfamily}, $ \gamma_{t}(M(K_n)) = \lceil 2n/3\rceil$.
	Moreover, by Theorem \ref{tocd2corona}, $  \gamma_{tc}(M(K_n\circ P_2)) \ge 2n+\gamma_{t}(M(K_n))$, and hence $\gamma_{tc}(M(K_n\circ P_2)) \ge 2n+\lceil 2n/3\rceil$.
	
	 If $n \equiv 0 \mod 3$, then consider 
	$$S=\{m_{12},m_{23},m_{45},m_{56},\dots, m_{(n-2)(n-1)}, m_{(n-1)n}\}.$$ 
	Similarly, if $n \equiv 1 \mod 3$, then consider 
	$$S=\{m_{12},m_{23},m_{45},m_{56},\dots, m_{(n-3)(n-2)},m_{(n-2)(n-1)}\}\cup\{m_{(n-1)n}\}.$$ 
	Finally, if $n \equiv 2 \mod 3$, then consider 
	$$S=\{m_{12},m_{23},m_{45},m_{56},\dots, m_{(n-4)(n-3)},m_{(n-3)(n-2)}\}\cup\{m_{(n-2)(n-1)},m_{(n-1)n}\}.$$ 
	Notice that in all three cases, $S$ is a total dominating set of $M(K_n)$ with $|S|=\lceil\frac{2n}{3}\rceil$. 
	Now since $D=\{ v_{2n+i}, m_{(n+i)(2n+i)} ~|~1\leq i \leq n\} \cup S$ is a total-outer connected dominating set of $M(K_n\circ P_2)$ with $|D|=2n+\gamma_{t}(M(K_n))$, we have $$ \gamma_{tc}(M(K_n\circ P_2)) = 2n+\lceil 2n/3\rceil.$$
\end{proof}
\begin{Remark}
	By Proposition \ref{tocdKnoP2corona}, the lower bound in Theorem \ref{tocd2corona} is tight.
\end{Remark}
As a consequence of Theorem \ref{tocd2corona} it is natural to study the following
\begin{Problem}
	Classify the graphs G of order $n \ge 3$ such that
	$$ \gamma_{tc}(M(G\circ P_2)) = 4n.$$
\end{Problem}
	A spider graph $S_{1,n,n}$ is obtained from the star graph $K_{1,n}$ by replacing every edge with a path of length $2$. 
Notice that $S_{1,1,1}$ is isomorphic to $K_{1,2}$ and $S_{1,2,2}$ to $P_5$. This implies that $\gamma _{tc}(M(S_{1,1,1}))=4$ by Corollary \ref{tocdK1n=2n-2}, and  $\gamma _{tc}(M(S_{1,2,2}))=6$ by Theorem \ref{tocdpn=2n-4}.

\begin{Theorem} \label{tocdS1,n,n=2n+2}
	For any spider graph $S_{1,n,n}$ on $2n+1$ vertices, with $n\geq 3$, we have
	$$\gamma _{tc}(M(S_{1,n,n}))=2n+2.$$
\end{Theorem}
\begin{proof}
	To fix the notation, assume that $V(S_{1,n,n})=\{v_0,v_1,\dots,v_{2n}\}$ and $E(S_{1,n,n})=\{v_0v_i,v_iv_{n+i}~|~1\le i\le n\}$. 
	Then $V(M(S_{1,n,n}))=V(S_{1,n,n})\cup M$, where $M=\{m_{0i}, m_{i(n+i)}~|~1 \leq i\leq n\}$. 
	
	Let $D$ be a total-outer connected dominating set of $M(S_{1,n,n})$. Since $N_{M(S_{1,n,n})}(v_{n+i})\cap D=\{m_{i(n+i)}\}$ for every $1 \le i \le n$, then $\{ m_{i(n+i)}~|~1 \leq i\leq n\} \subseteq D$. This implies that $v_{n+i} \in D$ for every $1 \le i \le n$, because $M(S_{1,n,n})-D$ is connected. Hence $\{ m_{i(n+i)}, v_{n+i}~|~1 \leq i\leq n\} \subseteq D$.
	Obviously $N_{M(S_{1,n,n})}(v_{0})\cap D \subseteq \{m_{0i}~|~1 \leq i\leq n\}$, so there exists at least a vertex $m_{0i} \in D$ for some $1 \leq i\leq n$. Moreover, since $m_{0i},m_{i(n+i)} \in D$ for some $i$, then we have that $v_i \in D$ because $M(S_{1,n,n})-D$ is connected. As a consequence, $|D| \ge 2n+2$.
	
	Now since $ D=\{ m_{i(n+i)}, v_{n+i}~|~1 \leq i\leq n\} \cup \{m_{01},v_{1}\} $ is a total-outer connected dominating set of $M(S_{1,n,n})$ with $|D| = 2n+2$ we have 
	$\gamma _{tc}(M(S_{1,n,n}))=2n+2.$
\end{proof}
We will now study the friendship graph $F_n$ of order $2n+1$ that is obtained by joining $n$ copies of the cycle graph $C_3$ with a common vertex.
\begin{Theorem}\label{toutfriendship}
	Let $F_n$ be the friendship graph with $n\ge2$. Then $$\gamma _{tc}(M(F_n))= 2n+1.$$
\end{Theorem}
\begin{proof} To fix the notation, assume $V(F_n)=\{v_0,v_1,\dots, v_{2n}\}$ and $E(F_n)=\{v_0v_1,v_0v_2,\dots, v_0v_{2n}\}\cup\{v_1v_2, v_3v_4,\dots,v_{2n-1}v_{2n}\}$. Then $V(M(F_n))=V(F_n)\cup \mathcal{M}$, where $\mathcal{M}=\{ m_{0i}~|~1\leq i \leq 2n \}\cup\{ m_{i(i+1)}~|~1\leq i \leq 2n-1 \text{ and } i \text{ is odd}\}$. 
	
	Let $D$ be a total outer-connected dominating set of $M(F_n)$. Set $D_i=V(M(F_n))\setminus\{v_0\}$.  
	Since $D$ is a total dominating set of $M(F_n)$, we have $|D \cap D_i| \ge 2$, because, otherwise there exists at least a vertex $v \in D_i$ such that $N_{M(F_n)}(v) \cap D=\emptyset$ which is a contradiction. Since for every $i,j$ such that $1\leq i,j \leq 2n-1$ and $i,j$ are odd $D_i \cap D_j = \emptyset$, we have $|D \cap (\bigcup\limits_{i=1}^{2n-1} D_i)| \ge 2n$ where $1\leq i \leq 2n-1$ and $i$ is odd. Without loss of generality, let $|D \cap (\bigcup\limits_{i=1}^{2n-1} D_i)| = 2n$. This implies that $v_{0} \notin D$. Obviously, $N_{M(F_n)}(v_{0}) \cap D \subseteq \{m_{0i}~|~1\leq i \leq 2n\}$. Without loss of generality, we may assume that $m_{01} \in N_{M(F_n)}(v_{0}) \cap D$. Then  $|D \cap D_1| =2$ and $N_{M(F_n)}(m_{01})\cap D \subseteq \{v_{1}, m_{12}, m_{02}\}$, implies that if $v_1 \in D_1$, then $N_{M(F_n)}(v_{2}) \cap D=\emptyset$, which is a contradiction. otherwise, if $m_{12} \in D_1$ or $m_{02} \in D_1$, then $M(F_n)-D$ is disconnected, which is a contradiction. Hence $|D| \ge 2n+1$. Therefor $\gamma _{tc}(M(F_n))\ge 2n+1$.
	\vskip 0.02cm
	 Now since  $D=\{v_i,  m_{i(i+1)}~|~1\leq i \leq 2n-1 \text{ and } i \text{ is odd}\}\cup\{m_{01}\}$ is a total outer-connected dominating set for $M(F_n)$ with $|D|=2n+1$, we have $\gamma _{tc}(M(F_n))= 2n+1$.  
\end{proof}
\section{Nordhaus-Gaddum-like relations}

Finding a Nordhaus-Gaddum-like relation for any parameter in graph theory is one of the traditional results which started after the work of Nordhaus and Gaddum \cite{Nordhaus}.


\vskip 0.02cm

As an immediate consequence of Theorem \ref{theo:TOCD<n+m-1}, we have the following result.
\begin{Proposition} \label{gamma _{tc}(M(overline{G}))}
Let $G$ be a connected graph of order $n \ge 4$ such that $\overline{G}$ is a connected graph, where $\overline{G}$ is the complement of $G$. Then $$2|\leaf(\overline{G})| \le\gamma _{tc}(M(\overline{G}))\le n+\dfrac{n(n-1)}{2}-m-1$$
where $n=|V(G)|$ and $m=|E(G)|$.	
\end{Proposition}
By Theorem \ref{theo:TOCD<n+m-1} and Proposition \ref{gamma _{tc}(M(overline{G}))} we have the following result.
\begin{Theorem} \label{nordGANDGCOMP}
	Let $G$ be a connected graph of order $n \ge 4$	such that $\overline{G}$ is a connected graph. Then 
	$$2|\leaf(G)|+2|\leaf(\overline{G})| \le \gamma _{tc}(M(G))+\gamma _{tc}(M(\overline{G}))\le \dfrac{n^{2}+3n-4}{2}.$$
\end{Theorem}
%
If we assume that the graph is a tree, we can improve on Theorem~\ref{nordGANDGCOMP}. Notice that if $T=K_{1,n-1}$, then $\overline{T}$ is not connected.
By Theorem \ref{T neq K_{1,n-1}} and Proposition \ref{gamma _{tc}(M(overline{G}))} we have the following result.
\begin{Theorem}\label{nordtree}
Let $T$ be a tree of order $n \ge 4$ such that $\overline{T}$ is connected. Then 
$$2|\leaf(T)|+2|\leaf(\overline{T})| \le \gamma _{tc}(M(T))+\gamma _{tc}(M(\overline{T}))\le \dfrac{n^{2}+3n-8}{2}.$$	
\end{Theorem}	
\begin{Remark}
Since $\overline{P_4}$ is isomorphic to $P_4$, by Theorem \ref{tocdpn=2n-4}, the lower bounds in Theorems \ref{nordGANDGCOMP} and \ref{nordtree} are tight when $G=P_4$.
\end{Remark}

\end{document}